\documentclass[12pt]{article}
\usepackage{amsmath, amsthm, amssymb,latexsym,enumerate}
\usepackage{graphicx}
\usepackage{hyperref}
\usepackage{xcolor}
\usepackage{todonotes}
\usepackage{cite}
\usepackage{dsfont}
\usepackage{authblk}
\usepackage{comment}
\usepackage{hyperref}
\usepackage{tikz}
\usetikzlibrary{calc}
\hypersetup{
  pdftitle={Independent domination of graphs with bounded maximum degree},
  pdfauthor={Eun-Kyung Cho, Jinha Kim, Minki Kim, and Sang-il Oum}}

\newcommand\fl[1]{\lfloor #1\rfloor}

\newcommand\abs[1]{\lvert #1\rvert}
\newtheorem{theorem}{Theorem}[section]
\newtheorem{lemma}[theorem]{Lemma}
\newtheorem{proposition}[theorem]{Proposition}
\newtheorem{corollary}[theorem]{Corollary}
\newtheorem{claim}[theorem]{Claim}

\numberwithin{equation}{section}
\newenvironment{subproof}[1][\proofname]{%
  \begin{proof}[#1]%
}{%
  \end{proof}%
}

\title{Independent domination of graphs with bounded maximum degree}
\author{Eun-Kyung Cho\thanks{Supported by Basic Science Research Program through the National Research Foundation of Korea (NRF) funded by the Ministry of Education (NRF-2020R1I1A1A0105858711).}}
\affil[1]{\small Department of Mathematics, Hankuk University of Foreign Studies, Yongin,~South~Korea.}
\author[2]{Jinha Kim\thanks{Supported by the Institute for Basic Science (IBS-R029-C1).}}
\affil[2]{\small Discrete Mathematics Group, Institute~for~Basic~Science~(IBS), Daejeon,~South~Korea.}
\author[$\dagger$3]{Minki Kim\thanks{Supported by Basic Science Research Program through the National Research Foundation of Korea (NRF) funded by the Ministry of Education (NRF-2022R1F1A1063424) and by GIST Research Project grant funded by the GIST in 2022.}}
\affil[3]{\small Division of Liberal Arts and Sciences, GIST, Gwangju, South Korea.}
\author[$\dagger$2,4]{Sang-il~Oum}
\affil[4]{\small Department of Mathematical Sciences, KAIST,  Daejeon,~South~Korea.}
\affil[ ]{E-mail address: \texttt{ekcho2020@gmail.com}, \texttt{jinhakim@ibs.re.kr}, \texttt{minkikim@gist.ac.kr}, \texttt{sangil@ibs.re.kr}}
\date\today

\begin{document}

\maketitle
\begin{abstract}
    An independent dominating set of a graph, also known as a maximal independent set, 
    is a set $S$ of pairwise non-adjacent vertices such that every vertex not in $S$ is adjacent to some vertex in $S$.
    We prove that 
    for $\Delta=4$ or $\Delta\ge 6$, every connected $n$-vertex graph of maximum degree at most $\Delta$ has an independent dominating set of size at most $(1-\frac{\Delta}{
        \lfloor\Delta^2/4\rfloor+\Delta
    })(n-1)+1$. In addition, we characterize all connected graphs having the equality
    and we show that other connected graphs have an independent dominating set of size at most $(1-\frac{\Delta}{
        \lfloor\Delta^2/4\rfloor+\Delta
    })n$.
\end{abstract}

\section{Introduction}
All graphs in this paper are simple, meaning that they have no loops and no parallel edges.
For a graph $G$, we write $V(G)$ for its vertex set and $E(G)$ for its edge set. 
For $v \in V(G)$, we use $d_G(v)$ to denote the degree of~$v$ in $G$.
For $v \in V(G)$, we use $N_G(v)$ to denote the set of neighbors of~$v$ in $G$, and $N_G[v] = N_G(v) \cup \{v\}$.
For $X \in V(G)$, we use $G-X$ to denote the subgraph of~$G$ induced by $V(G) \setminus X$.
We write $\Delta(G)$ to denote the maximum degree of vertices in a graph $G$.
A vertex of degree $0$ is called \emph{isolated}.

An \emph{independent dominating set} of a graph is a set $S$ of vertices that are pairwise non-adjacent, which is called \emph{independent}, and every vertex not in $S$ has a neighbor in $S$, which is called \emph{dominating}.
Independent dominating sets appear as early as 1962
in Berge~\cite[Chapter 5]{Berge1962} and Ore~\cite[Chapter 13]{ore1962theory} under the names \emph{kernels} and maximal independent sets, respectively.
It is easy to see that independent dominating sets are precisely 
maximal independent sets, which is stated in both Berge~\cite[Chapter 5]{Berge1962} and Ore~\cite[Chapter 13]{ore1962theory}.
We write $i(G)$ to denote the size of the minimum independent dominating set in a graph $G$.
Haynes, Hedetniemi, and Slater discussed the relation between the domination parameters and maximum degree in their book~\cite[Subsection 9.3.2]{HHS1998}.

We are interested in bounding $i(G)$ in terms of the number of vertices in a graph~$G$ of maximum degree at most $\Delta$.
It is trivial to see the inequality \[ i(G)\le \abs{V(G)}-\Delta\] 
because a maximal independent set containing a vertex of maximum degree has at most $\abs{V(G)}-\Delta$ vertices. 
Domke, Dunbar, and Markus~\cite{DDM1997} showed the above inequality 
and 
determined connected bipartite graphs and trees having the equality.
As a refinement of the trivial inequality, Blidia, Chellali, and Maffray \cite{blidia2006extremal} showed that $i(G)\le \abs{V(G)}-\Delta'(G)$, where $\Delta'(G)=\max_{v\in V(G)}(d_G(v)+\nu(G-N_G[v]))$
and $\nu(G-N_G[v])$ is the size of a largest matching in $G-N_G[v]$.

Akbari~et~al.~\cite{AADHHN2020} 
showed that
$i(G)\le \frac{1}{2}\abs{V(G)}$
for graphs $G$ of maximum degree at most $3$ with no isolated vertex
and characterized all graphs achieving the equality.
Cho, Choi, and Park~\cite[Theorem 1.5]{cho2021independent} proved that 
$i(G)\le \frac{5}{9}\abs{V(G)}$
for graphs $G$ of maximum degree at most $4$ with no isolated vertices.

Our main theorem provides a tight upper bound of $i(G)$ 
in terms of $\abs{V(G)}$ 
for graphs $G$ of maximum degree at most $\Delta$.

\begin{theorem}\label{thm:form1}
    Let $\Delta\ge 4$ be an integer and 
    let $G$ be a connected graph of maximum degree at most $\Delta$.
    \begin{itemize}
        \item If $\Delta\neq 5$, then $G$
        has an independent dominating set of size at most 
        $\bigl(1-\frac{\Delta}{\lfloor \Delta^2/4\rfloor+\Delta}\bigr) (\abs{V(G)}-1)+1$.
        \item If $\Delta=5$, then $G$
        has an independent dominating set of size at most 
        $\max\left(\frac{5}{9}\abs{V(G)}, \bigl(1-\frac{\Delta}{\lfloor \Delta^2/4\rfloor+\Delta}\bigr) (\abs{V(G)}-1)+1\right)$.
    \end{itemize}
\end{theorem}
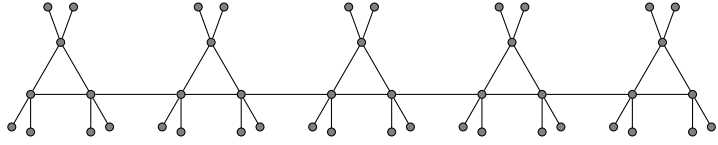
\begin{figure}
    \begin{center}
    \begin{tikzpicture}
        \tikzstyle{v}=[circle,draw,fill=black!50,inner sep=0pt,minimum width=3pt]
        \foreach \i in {0,1,2,3,4} {
            \draw (2*\i,0) node[v] (v\i0){} --+ (0:.8) node[v](v\i1){} --+ (60:.8) node[v](v\i2){} --(v\i0);
            \draw (v\i0)--+ (240:.5) node[v]{};
            \draw (v\i0)--+ (270:.5) node[v]{};
            \draw (v\i1)--+ (270:.5) node[v]{};
            \draw (v\i1)--+ (300:.5) node[v]{};
            \draw (v\i2)--+ (90-20:.5) node[v]{};
            \draw (v\i2)--+ (90+20:.5) node[v]{};
        }
        \draw (v01)--(v10);
        \draw (v11)--(v20);
        \draw (v21)--(v30);
        \draw (v31)--(v40);
    \end{tikzpicture}
    \end{center}    
    \caption{A connected graph $G$ satisfying $i(G)=\frac{5}{9}\abs{V(G)}$ for $\Delta=5$.}\label{fig:ex}
\end{figure}

In addition, we characterize all connected graphs achieving the equality of Theorem~\ref{thm:form1} when $\Delta\neq5$. They will be called $\Delta$-special graphs. 
For an integer $\Delta\ge4$, 
a graph is \emph{$\Delta$-special} if  it is isomorphic to one of the following.
\begin{enumerate}[(a)]
    \item $K_1$.
    \item A graph obtained from $K_{\lceil \Delta/2\rceil+1}$ or $K_{\lfloor \Delta/2\rfloor+1}$ 
by attaching degree-$1$ vertices so that each vertex in the initial clique has degree $\Delta$.
    \item A graph obtained from an odd cycle by attaching two vertices of degree~$1$ at each vertex when $\Delta=4$.
    \item A graph obtained from the union of two $\Delta$-special graphs $G_1$ and $G_2$ with $\abs{V(G_1)\cap V(G_2)}=1$
    where the vertex in the intersection does not belong to any cycles in $G_1$ or $G_2$.
\end{enumerate}
A precise and equivalent definition will be discussed in Section~\ref{sec:special}.

Now we present a stronger theorem, implying Theorem~\ref{thm:form1}.
\begin{theorem}\label{thm:form2}
    Let $\Delta\ge4$ be an integer and let $G$ be a connected graph of maximum degree at most $\Delta$.
\begin{enumerate}[(i)]
    \item     If $G$ is $\Delta$-special, then 
    $i(G)=\left(1-\frac{\Delta}{\lfloor \Delta^2/4\rfloor+\Delta}\right) \abs{V(G)}+\frac{\Delta}{\lfloor \Delta^2/4\rfloor+\Delta}$.
    \item If $G$ is not $\Delta$-special and $\Delta\neq 5$, then 
    $i(G)\le \left(1-\frac{\Delta}{\lfloor \Delta^2/4\rfloor+\Delta}\right) \abs{V(G)}$.
    \item If $G$ is not $\Delta$-special and $\Delta=5$, then 
    $i(G)\le \frac{5}{9}\abs{V(G)}$.
\end{enumerate}
\end{theorem}
There is an infinite family of non-special connected graphs achieving the bound of Theorem~\ref{thm:form2}. 
When $\Delta\neq 5$, let $G$ be a graph obtained from any non-trivial $\Delta$-special graph by removing one vertex of degree $1$. It is straightforward to verify that 
$i(G)=\left(1-\frac{\Delta}{\lfloor \Delta^2/4\rfloor+\Delta}\right)\abs{V(G)}$.

For $\Delta=5$, let $G$ be a graph obtained from a path $v_1v_2\ldots v_{3k}$ on $3k$ vertices by adding $k$ edges $v_{3i-2}v_{3i}$ for all $i=1,2,\ldots,k$
and attaching $2$ vertices of degree $1$ to each vertex on the path, see Figure~\ref{fig:ex}.
Let $S$ be a set consisting of $v_3$, $v_6$, $\ldots$, $v_{3k}$ and all vertices of degree $1$ adjacent to $v_{i}$ for some $i\not\equiv 0\pmod3$.
Then it is easy to verify that $S$ is a minimum independent dominating set of $G$
and $i(G)=\abs{S}=5k$ and therefore $i(G)=\frac{5}{9}\abs{V(G)}$.

Theorem~\ref{thm:form2} proves the following corollary, which was conjectured by Cho, Choi, and Park~\cite{cho2021independent}. They claimed to have verified the conjecture for $\Delta\le 8$ and presented its proof for $\Delta=4$.
Let $H(p, q)$ be the graph obtained by attaching $q$~vertices of degree $1$ to every vertex of $K_p$.
\begin{corollary}\label{cor:main}
    Let $\Delta$ be a positive integer.
    Every graph $G$ with maximum degree at most $\Delta$ and no isolated vertices
    has an independent dominating set having at most 
    $(1-\frac{\Delta}{
        \lfloor(\Delta+2)^2/4\rfloor
    }) \abs{V(G)}$ vertices.
    Furthermore, when $\Delta\ge 4$, 
    the equality holds if and only if 
    every component is isomorphic to $H(\lceil \Delta/2\rceil+1, \lfloor \Delta/2\rfloor)$
    or $H(\lfloor \Delta/2\rfloor+1,\lceil \Delta/2\rceil)$.
\end{corollary}
\begin{proof}[Proof of Corollary~\ref{cor:main} assuming Theorem~\ref{thm:form2}]
    Let $n=\abs{V(G)}$.
    It is trivial if $\Delta\le2$.
    We may assume that $G$ is connected.
    Akbari et~al.~\cite[Theorem 3]{AADHHN2020} showed that $i(G)\le n/2$ if $\Delta=3$.
    So we may assume that $\Delta\ge 4$.
    We may assume that $G$ is $\Delta$-special by (ii) and (iii) of Theorem~\ref{thm:form2}, because $\frac59<\frac7{12}$ and $\frac{\Delta}{\lfloor\Delta^2/4\rfloor+\Delta}>\frac{\Delta}{
        \lfloor(\Delta+2)^2/4\rfloor
    }$.
    Since $G$ has no isolated vertex,
    $n\ge (\lfloor \Delta/2\rfloor+1)(\Delta-\lfloor \Delta/2\rfloor+1)$
    or 
    $n\ge (\lceil \Delta/2\rceil+1)(\Delta-\lceil  \Delta/2\rceil+1)$.
    In both inequalities, the right-hand side is equal and we have 
    \[ n\ge \lfloor (\Delta+2)^2/4\rfloor.\] 
    This is equivalent to 
    $\left(1-\frac{\Delta}{\lfloor \Delta^2/4\rfloor+\Delta}\right) n+ \frac{\Delta}{\lfloor \Delta^2/4\rfloor+\Delta}
    \le (1-\frac{\Delta}{
        \lfloor(\Delta+2)^2/4\rfloor
    })n$, proving the inequality.
    Furthermore the equality holds if and only if $G$ is $\Delta$-special and $n=\lfloor (\Delta+2)^2/4\rfloor$. Indeed, $H(\lceil \Delta/2\rceil+1, \lfloor \Delta/2\rfloor)$
    and $H(\lfloor \Delta/2\rfloor+1,\lceil \Delta/2\rceil)$ are the only such graphs. This completes the proof.
\end{proof}

Our proof is by induction on  $\abs{V(G)}+\abs{E(G)}$.
Here are two key inequalities on $i(G)$ for our proof.
\begin{enumerate}[(1)]
    \item $i(G)\le i(G-N_{G}[v])+1$ for a vertex $v$ of a graph $G$.
    \item $i(G)\le i(H)$ if $H$ is obtained from a graph $G$ by deleting all but one of the edges incident with a fixed vertex of degree at least $2$.
\end{enumerate}
The inequality (1) is commonly used when dealing with the independent domination, 
but the authors could not find papers using or mentioning (2). Perhaps it is because it is difficult to control the change of $i(G)$ after deleting edges. 

The paper is organized as follows. Section~\ref{sec:lem} presents several useful lemmas, including the proof of (2). 
Section~\ref{sec:special} discusses $\Delta$-special graphs.
Section~\ref{sec:main} presents the proof of the main theorem.

\section{Several lemmas}\label{sec:lem}

We write $n_0(G)$ to denote the number of isolated vertices of a graph~$G$.
We start with a few useful lemmas.

\begin{lemma}\label{lem:deg}
    Let $H$ be a graph and $w$ be a vertex with degree at least $2$.
    Let $u_1,u_2,\ldots,u_k$ be the list of all neighbors of $w$.
    Then $H'=H-\{wu_i:i \in \{1,2,\ldots,k-1\}\}$ satisfies $i(H) \leq i(H')$.
    \end{lemma}
    \begin{proof}
    Let $S$ be a minimum independent dominating set of $H'$.
    Then $S$ is a dominating set of $H$.
    If $S$ is an independent set of $H$, then $i(H)\le \abs{S}=i(H')$
    and therefore we may assume that 
    $S$ contains both $w$ and $u_j$ for some $j\in\{1,2,\ldots,k-1\}$.
    If $S \setminus \{w\}$ has a neighbor of $u_k$, then $S \setminus \{w\}$ is an independent dominating set of $H$ so that $i(H) \le i(H')-1$. Therefore, we may assume that $S \setminus \{w\}$ has no neighbor of $u_k$.
    It follows that $(S \setminus \{w\}) \cup \{u_k\}$ is an independent dominating set of $H$, implying that $i(H)\le \abs{(S \setminus \{w\}) \cup \{u_k\}}=i(H')$.
    \end{proof}

\begin{lemma}\label{lem:1sum}
    Let $G_1$, $G_2$ be graphs with the unique common vertex $v$.
    If $G_i$ has a minimum independent dominating set $S_i$ containing  $v$ such that $S_i\setminus \{v\}$ is a dominating set of $G_i-v$ for some $i=1,2$, 
    then 
    \[ i(G_1\cup G_2)=i(G_1)+i(G_2)-1.\] 
\end{lemma}
\begin{proof}
    Let $G=G_1\cup G_2$.
    First let us prove $i(G) \ge i(G_1)+i(G_2)-1$.
    Let $S$ be an independent dominating set of $G$.
    If $S\cap V(G_1)$ and $S\cap V(G_2)$ are independent dominating sets of $G_1$ and $G_2$, respectively, then 
    \begin{align*} \abs{S}&\ge \abs{S\cap V(G_1)}+\abs{S\cap V(G_2)}-1
    \ge i(G_1)+i(G_2)-1.
    \end{align*}
    Note that we subtracted $1$ because $v$ may belong to $S$.
    Thus we may assume that $S\cap V(G_1)$ is not an independent dominating set of $G_1$. Then $v\notin S$ and $v$ is not dominated by $S\cap V(G_1)$. 
    Then 
    $(S\cap V(G_1))\cup\{v\}$ is an independent dominating set of $G_1$ and therefore $\abs{S\cap V(G_1)} \ge i(G_1)-1$.
    In this case, $S\cap V(G_2)$ dominates $v$ and therefore 
    $\abs{S\cap V(G_2)} \ge i(G_2)$. 
    This again implies that 
    \[ \abs{S}=\abs{S\cap V(G_1)}+\abs{S\cap V(G_2)}
    \ge i(G_1)+i(G_2)-1.\] 
    
    Now let us prove that $i(G)\le i(G_1)+i(G_2)-1$.
    Assume that $G_2$ has a minimum independent dominating set $S_2$ containing $v$ such that 
    $S_2\setminus\{v\}$ is a dominating set of $G_2-v$.
    Let $S_1$ be a minimum independent dominating set of $G_1$.
    Then $S=S_1 \cup (S_2\setminus \{v\})$ is an independent dominating set of $G$ and therefore 
    $i(G)\le \abs{S_1}+\abs{S_2}-1=i(G_1)+i(G_2)-1$.
\end{proof}

\begin{lemma}\label{lem:all-d1}
    Let $\Delta\ge 4$, $d_1<\Delta$ be integers.
    Let $G$ be a connected graph of maximum degree at most $\Delta$.
    Let $X$ be the set of all vertices of degree~$1$.
    If every vertex of degree at least $2$ is adjacent to 
    exactly $d_1$ neighbors of degree~$1$, 
    then 
    \[ 
        i(G)\le \left(1-\frac{\Delta-1}{(d_1+1)(\Delta-d_1)}\right)\abs{V(G)}, 
    \] 
    unless $G-X$ is a complete graph on $\Delta-d_1+1$ vertices
    or is an odd cycle when $\Delta-d_1=2$.
\end{lemma}
\begin{proof}
    Let $n=\abs{V(G)}$.
    Let $m=n/(d_1+1)$ be the number of vertices of $G-X$.
    Observe that $G-X$ is connected.
    Let $S$ be a maximum independent set of $G-X$
    and let $\alpha=\abs{S}$.
    Then $\alpha\ge \frac{m}{\chi(G-X)}=\frac{1}{(d_1+1)\chi(G-X)}n$. 
    Observe that $S\cup\{v\in X: v\notin N_G(S)\}$ is an independent dominating set of $G$ and therefore 
    \begin{align*}
        i(G)&\le \alpha + (m-\alpha)d_1 
          =md_1-\alpha(d_1-1) \\
        &\le  n-\frac{d_1-1+\chi(G-X)}{(d_1+1)\chi(G-X)}n.
    \end{align*} 

    By Brooks' theorem, we deduce that $\chi(G-X)\le \Delta-d_1$. Therefore 
    \[ 
        i(G)
        \le  m(d_1+1)-m\frac{\Delta-1}{\Delta-d_1}
        = \left(1-\frac{\Delta-1}{(d_1+1)(\Delta-d_1)}\right) n. 
        \qedhere
    \] 
\end{proof}
\begin{lemma}\label{lem:select}
    Let $\Delta$ be an integer.
    Let $G$ be a connected graph with at least one vertex of degree at least $2$ such that the maximum degree is at most $\Delta$ and every vertex of degree at least $2$ is adjacent to at least one vertex of degree~$1$.    
    Let $d_1$ be the maximum number of degree-$1$ neighbors of a vertex.
    If some vertex has less than $d_1$ neighbors of degree $1$, 
    then there is a vertex $v$ such that
    \[ 
        \frac{n_0(G-N_G[v])+1}{d_G(v)}\le \frac{\lfloor \Delta^2/4\rfloor}{\Delta}.
    \]
\end{lemma}
\begin{proof}
    Let $X$ be the set of all vertices of degree~$1$ in $G$.
    Since $G-X$ is connected, there is an edge $uv$ such that
    $u$ has less than $d_1$ neighbors of degree~$1$
    and $v$ has exactly $d_1$ neighbors of degree~$1$.
    By the assumption, every isolated vertex of $G-N_G[v]$ is adjacent to exactly one neighbor of $v$
    and therefore each isolated vertex of $G-N_G[v]$ is a degree-$1$ neighbor of a vertex in $N_G(v)$.
    Since $u$ has less than $d_1$ neighbors of degree~$1$, we have 
    \[ n_0(G-N_G[v])+1\le d_1(d_G(v)-d_1).\]
    Let $f(x,y)=\frac{x(y-x)}{y}$ for integers $1\le x<y\le \Delta$.
    Observe that $f(x,y)\le f(x,y+1)$ for $x<y<\Delta$
    and therefore the maximum of $f$ is achieved when $y=\Delta$. 
    Since $\frac{n_0(G-N_G[v])+1}{d_G(v)}\le f(d_1,d_G(v))$ and $x(\Delta-x)\le \lfloor \Delta^2/4\rfloor$ for any integer $x$, we deduce that 
    $ \frac{n_0(G-N_G[v])+1}{d_G(v)}\le \frac{\lfloor \Delta^2/4\rfloor}{\Delta}$.
\end{proof}
\begin{lemma}\label{lem:ineq}
    Let $x<y$ be positive integers and $y\ge 5$. Then 
    \[ 
    \frac{x(y-x)+1}{y}< \frac{x(y+1-x)}{y+1}
    \text{ or }
    \frac{x(y-x)+1}{y}< \frac{(x+1)(y-x)}{y+1}.
    \] 
\end{lemma}
\begin{proof}
    It is enough to show that 
    \[ 
    2\frac{x(y-x)+1}{y}<  \frac{x(y+1-x)}{y+1}+\frac{(x+1)(y-x)}{y+1}, 
    \] 
    which is equivalent to 
    \[ 
    \frac{2xy-2x^2+2}{y} < \frac{ 2xy-2x^2 +y   }{y+1}.
    \] 
    As $y\ge 5$, we have 
    \begin{align*}
        \lefteqn{(y+1)(2xy-2x^2+2)-y(2xy-2x^2+y)}\\
        &= 2y+2xy-2x^2+2 - y^2 \\
        &= 2y+2  -x^2-(y-x)^2\\
        &\le 2y+2 - \frac{1}{2} y^2\\
        & < 0
    \end{align*}       
    by Cauchy-Schwarz inequality.
\end{proof}
\begin{lemma}\label{lem:select2}
    Let $\Delta\ge 4$, $d_1>0$ be integers.
    Let $G$ be a connected graph with at least one vertex of degree at least $2$ such that the maximum degree is at most $\Delta$
    and every vertex of degree at least~$2$ has exactly $d_1$ neighbors of degree~$1$.
    Let $v$ be a vertex of $G$ of degree at least~$2$.
    If $d_1\notin\{\lceil \Delta/2\rceil,\lfloor \Delta/2\rfloor\}$
    or $d_G(v)<\Delta$, 
    then 
    \[ 
        \frac{n_0(G-N_G[v])+1}{d_G(v)}\le \frac{\lfloor \Delta^2/4\rfloor}{\Delta},
    \]
    unless $\Delta=5$, $d_G(v)=4$, and $d_1=2$.
\end{lemma}
\begin{proof}
    Let $v$ be a vertex of degree at least $2$. 
    By the assumption, every isolated vertex of $G-N_G[v]$ is adjacent to exactly one neighbor of $v$
    and therefore each isolated vertex of $G-N_G[v]$ is a degree-$1$ neighbor of a vertex in $N_G(v)$.
    Therefore,
    \[ n_0(G-N_G[v])=d_1(d_G(v)-d_1).\]

    If $d_G(v)=\Delta$, 
    then 
    $d_1(\Delta-d_1)\le \lfloor \Delta^2/4\rfloor-1$ because $d_1\notin\{\lceil \Delta/2\rceil,\lfloor \Delta/2\rfloor\}$. Therefore
    $ (n_0(G-N_G[v])+1)/d_G(v) \le \lfloor \Delta^2/4\rfloor / \Delta$.
    So we may assume that $d_G(v)<\Delta$.

    If $d_1=d_G(v)-1$ or $d_1=1$, then $\frac{d_1(d_G(v)-d_1)+1}{d_G(v)}=1\le \frac1\Delta\lfloor \Delta^2/4\rfloor$.
    Thus we may assume that $1<d_1<d_G(v)-1$.
    So  $d_G(v)\ge 4$. If $d_G(v)=4$, then $d_1=2$ and $\frac{d_1(d_G(v)-d_1)+1}{d_G(v)}=\frac{5}{4}$. If $\Delta\ge 6$, then $\frac1\Delta\lfloor {\Delta^2/4}\rfloor\ge \frac54=\frac{d_1(d_G(v)-d_1)+1}{d_G(v)}$. 
    (If $\Delta=5$, then $\frac{\lfloor \Delta^2/4\rfloor}{\Delta}=\frac65$, $\frac{d_1(d_G(v)-d_1)+1}{d_G(v)}=\frac54$.)

    Thus we may assume that $5\le d_G(v)<\Delta$.
    Then by Lemma~\ref{lem:ineq}, 
    $\frac{n_0(G-N_G[v])+1}{d_G(v)}\le \frac{(d_1+1)(d_G(v)-d_1)}{d_G(v)+1}$
    or 
    $\frac{n_0(G-N_G[v])+1}{d_G(v)}\le \frac{d_1(d_G(v)+1-d_1)}{d_G(v)+1}$.
    By applying Lemma~\ref{lem:ineq} inductively, we deduce that there is a positive integer $x<\Delta$ such that
    $\frac{n_0(G-N_G[v])+1}{d_G(v)}\le \frac{x(\Delta-x)}{\Delta}$.
    We deduce the conclusion because $x(\Delta-x)\le \lfloor \Delta^2/4\rfloor$.
\end{proof}

\section{Special graphs}\label{sec:special}
For an integer $\Delta\ge 4$, we say that a connected graph $G$ is called \emph{$\Delta$-special} if it satisfies the following.
\begin{enumerate}[(i)]
    \item Every vertex that belongs to some cycle has degree $\Delta$.
    \item Every edge has at least one end that belongs to some cycle.
    \item Each component of the subgraph induced by the set of all vertices that belong to cycles 
    is isomorphic to one of the following.
    \begin{enumerate}
        \item $K_{\lceil \Delta/2\rceil+1}$.
        \item $K_{\lfloor \Delta/2\rfloor+1}$.
        \item An odd cycle when $\Delta=4$.
    \end{enumerate}
\end{enumerate}
When $\Delta$ is clear from the context, we will say \emph{special} instead of $\Delta$-special for brevity.
A special graph is called \emph{non-trivial} if it has at least two vertices and \emph{trivial} otherwise.

\begin{proposition}\label{prop:special}
    Let $\Delta\ge 4$ be an integer.
    If $G$ is a $\Delta$-special graph on $n$ vertices, then 
    \[ i(G)= \left(1-\frac{\Delta}{\lfloor \Delta^2/4\rfloor+\Delta}\right)(n-1)+1 
    .\] 
\end{proposition}
\begin{proof}
    Let $t=\frac{\Delta}{\lfloor \Delta^2/4\rfloor+\Delta}$.
    Let $X$ be the set of all vertices $v$ of $G$
    such that there is a cycle containing $v$.
    If $X=\emptyset$, then $n=i(G)=1$ and therefore we may assume that $X$ is nonempty.

    First, we will treat the case that $G[X]$ has only one component.
    Let $S$ be an independent dominating set of $G$. 
    If $G[X]$ has exactly one component, then 
    $G[X]$ is 
    isomorphic to 
    $K_{\lceil \Delta/2\rceil+1}$ or 
    $K_{\lfloor \Delta/2\rfloor+1}$, 
    unless $\Delta=4$ and it is isomorphic to an odd cycle.
    Let $r$ be an integer such that $G[X]$ is $r$-regular. 
    By the definition, $r=\lceil \Delta/2\rceil$ or $r=\lfloor \Delta/2\rfloor  $.    
    Furthermore, by (i) and (ii), 
    $G$ is isomorphic to a graph obtained from $G[X]$ 
    by attaching $\Delta-r$ vertices of degree $1$ to every vertex in $X$.
    Then $n=|X|(\Delta-r+1)$.
    Suppose that $G[X]$ is isomorphic to $K_{\lceil \Delta/2\rceil+1}$ or $K_{\lfloor \Delta/2\rfloor+1}$. 
    Then $\abs{S\cap X}\le 1$. 
    Note that $n=
    \lfloor \Delta/2+1\rfloor \lceil \Delta/2+1\rceil
    = \lfloor (\Delta+2)^2/4\rfloor=\lfloor \Delta^2/4\rfloor+\Delta+1$.
    Since $S$ is dominating, 
    \begin{align*} 
        \abs{S}&= \abs{S\cap X}+\abs{X\setminus S}(\Delta-r)
    \ge 1+(|X|-1)\cdot (\Delta-r)=\lfloor \Delta^2/4\rfloor +1 \\
    &=(1-t)(n-1)+1
    .
    \end{align*}
    In this case, the equality holds if and only if 
    $\abs{S\cap X}=1$ and such a set $S$ exists. 
    If $\Delta=4$ and 
    $G[X]$ is an odd cycle,
    then $\abs{S\cap X}\le \frac{\abs{X}-1}{2}$ and since $S$ is dominating, 
    \[ \abs{S}= \abs{S\cap X}+\abs{X\setminus S}(\Delta-r)
    \ge 
    \frac{\abs{X}-1}{2} + \frac{\abs{X}+1}{2}\cdot 2
    = \frac{3\abs{X}+1}{2}= \frac{n-1}{2}+1.\] 
    In this case, the equality holds if and only if 
    $\abs{S\cap X}=(\abs{X}-1)/2$ and such a set $S$ exists. 
    Thus $i(G)=(1-t)(n-1)+1$ if $G[X]$ has one component. 
    
    So we may assume that $G[X]$ has at least two components $C_1$ and $C_2$. 
    Since $G$ is connected, 
    $G$ has a shortest path from $C_1$ to $C_2$. 
    By (ii) and (iii), this path has length $2$
    whose middle vertex $v$ is not in $X$.
    By the definition of $X$, $G-v$ is disconnected.
    Among all choices of $v$ and a component $F$ of $G-v$, 
    we choose the one such that $F$ is minimal. 
    We claim that $G[V(F)\cap X]$ is connected.
    Suppose that $G[V(F)\cap X]$ contains two components $D_1$ and $D_2$. 
    Then $F$ has a shortest path from $D_1$ to $D_2$. 
    By (ii) and (iii), this path has length $2$
    whose middle vertex $w$ is not in $X$.
    By the definition of $X$, $F-w$ is disconnected.
    Since $v$ is not in a cycle, $v$ cannot have neighbors in distinct components of $F-w$. Let $F'$ be a component of $F-w$ not having neighbors of $v$. Then $F'$ is a component of $G-w$, contradicting the choice of $v$ and $F$.
    This proves that $G[V(F)\cap X]$ is connected.
    
    Let $G_1=G-V(F)$ and $G_2=G[V(F)\cup \{v\}]$.
    Observe that both $G_1$ and $G_2$ are special.    
    
    Note that $G_2[V(G_2)\cap X]$  is isomorphic to one of (a), (b), (c)
    and $v$ has degree $1$ in $G_2$. 
    Because of our analysis in the first part of the proof, 
    $G_2$ has a minimum independent dominating set $S_2$ containing $v$ and another vertex $w$ such that $N_{G_2}(v) \subseteq N_{G_2}(w)$. 
    Then $S_2\setminus \{v\}$ is a dominating set of $G_2-v$.
    By the induction hypothesis, $i(G_1)=(1-t)(\abs{V(G_1)}-1)+1$ 
    and $i(G_2)=(1-t)(\abs{V(G_2)}-1)+1$.
    By Lemma~\ref{lem:1sum}, $i(G)=i(G_1)+i(G_2)-1
    = (1-t)(\abs{V(G_1)}-1) 
    + (1-t)(\abs{V(G_2)}-1) + 1 
    = (1-t)(n-1)+1$.
\end{proof}

\section{Main theorem}\label{sec:main}

Let $n_\Delta(G)$ be the number of $\Delta$-special components.
We now prove our main theorem, which proves Theorem~\ref{thm:form2}.

\begin{theorem}\label{thm:main}
    Let $\Delta\ge 4$ be an integer and let $G$ be a graph with maximum degree at most~$\Delta$.
    Let \[ 
        t=\begin{cases}
            \frac{4}{9} &\text{if }\Delta=5,\\
            \frac{\Delta}{\fl{\Delta^2/4}+\Delta}&\text{otherwise.}
        \end{cases}
    \] 
    Then 
    $i(G)\le (1-t)\abs{V(G)}+tn_\Delta(G)$.
\end{theorem}
\begin{proof}
    We proceed by induction on $\abs{V(G)}+\abs{E(G)}$.
    By the induction hypothesis, we may assume that $G$ is connected.
    It is true if $G$ is special by Proposition~\ref{prop:special} and therefore we may assume that 
    $G$ is not special.
    Thus $n_\Delta(G)=0$.
    Note that if $\Delta=5$, then $\frac{\Delta}{\fl{\Delta^2/4}+\Delta}>\frac{4}{9}$.
    \begin{claim}\label{clm:degree1}
        If $G$ has a vertex of degree at least~$2$ 
        with no neighbors of degree~$1$, 
        then $i(G)\le (1-t)\abs{V(G)}$.
    \end{claim}
    \begin{subproof}
        Suppose that $G$ has a vertex $v$ of degree at least $2$
        with no neighbors of degree~$1$.
        Let $u_1$, $u_2$, $\ldots$, $u_k$ be the neighbors of~$v$. 
        For $j\in[k]$, 
        let $H_j:=G-\{v u_i : i \in [k], ~ i\neq j\}$.
        By Lemma~\ref{lem:deg}, $i(G)\le i(H_j)$.
        Since $d_G(u_i)>1$ for all $i\in [k]$, 
        $n_0(H_j)=0$ for all $j\in [k]$.
        Thus if $H_j$ has no non-trivial special component for some $j\in [k]$, then $i(G)\le i(H_j)\le (1-t)\abs{V(G)}$ by the induction hypothesis.
        Thus we may assume that $H_j$ has a non-trivial special component $C_j$ for all $j\in [k]$.
        Since $C_j$ is not a component of~$G$, 
        $C_j$ contains a vertex $x_j$ incident with a deleted edge in $G$.
        Since $x_j$ has degree less than $\Delta$ in $C_j$, 
        $C_j$ has no cycle containing $x_j$ by the definition of a special graph. By the definition of a special graph, 
        every neighbor of $x_j$ in $C_j$ belongs to some cycle of $C_j$.
    
        Suppose $x_j=u_i$ for some $j\in [k]$ and $i \in [k]\setminus\{j\}$.
        Let $F_j$ be the set of edges of $G-u_iv$ incident with $u_i$.
        As $u_i$ has degree at least $2$ in $G$, $F_j$ is nonempty.
        By Lemma~\ref{lem:deg}, $i(G)\le i(G-F_j)$. 
        By the induction hypothesis, we may assume that $G-F_j$ contains a non-trivial special component $C'_j$.        
        Since $C'_j$ is not a component of~$G$, there is a vertex $y_j$ of $C'_j$ that is incident with some edge in $F_j$.
        Since the degree of $y_j$ in $G-F_j$ is less than $\Delta$, 
        $y_j$ does not belong to any cycle of $C'_j$
        and therefore $y_j=x_j$.
        However in this case, $G=C_j\cup C_j'$, $u_i$ is the unique vertex in $C_j\cap C_j'$, and $u_i$ does not belong to any cycle in $C_j$ or $C_j'$. This implies that $G$ is special, contradicting the assumption.

     Thus $x_j=v$ and $V(C_j)\cap \{u_1,\ldots,u_{k}\}=\{u_j\}$ for all $j\in [k]$.
    Then $G$ is obtained from the disjoint union of $C_1$, $\ldots$, $C_k$ by identifying one vertex from each component that does not belong to a cycle to become $v$. 
    This implies that $G$ is special, contradicting the assumption. 
    \end{subproof}

    Suppose that $G-N_G[v]$ has a non-trivial special component~$C$ for some vertex $v$.
    Since $G$ is connected, $C$ has a vertex~$x$ adjacent to some neighbors of~$v$
    and therefore the degree of $x$ in $C$ is less than $\Delta$.
    By the definition of a special graph, 
    all neighbors of $x$ in $C$ have degree $\Delta$
    and therefore all neighbors of~$x$ in~$G$ have degree at least $2$ in $G$, implying 
    that $i(G)\le (1-t)\abs{V(G)}$ by  
    Claim~\ref{clm:degree1}.
    Therefore, we may assume that $G-N_G[v]$  has no non-trivial special component for every vertex $v$, that is, $n_{\Delta}(G-N_G[v])=n_0(G-N_G[v])$.

    Let $d_1$ be the maximum number of degree-$1$ neighbors of a vertex of $G$.
    By Claim~\ref{clm:degree1}, we may assume that $G$ has a vertex of degree $1$
    and therefore $d_1>0$.
    Suppose that there is a vertex of degree at least $2$ having less than $d_1$ neighbors of degree~$1$.
    Let $v$ be the vertex chosen by Lemma~\ref{lem:select}.
    By the induction hypothesis, we have
    \begin{align*}
        i(G)
        &\le i(G-N_G[v])+1 \\
        &    
        \leq (1-t)(\abs{V(G)}-d_G(v)-1)+t n_0(G-N_G[v])+1\\
        &=(1-t)\abs{V(G)} -(1-t)d_G(v)+tn_0(G-N_G[v])+t \\
        &\le (1-t)\abs{V(G)} -d_G(v)+ t \frac{\lfloor \Delta^2/4\rfloor+\Delta}{\Delta} d_G(v)
        &\text{by Lemma~\ref{lem:select}}\\
        &\le (1-t)\abs{V(G)}.
    \end{align*}

    Thus we may assume that every vertex of degree at least $2$ has exactly $d_1$ neighbors of degree~$1$.
    Let $X$ be the set of all vertices of degree~$1$.
    Since $G$ is connected, $G-X$ is connected.
    Note that $\abs{V(G-X)}=\frac{1}{d_1+1}\abs{V(G)}$.

   If $\Delta=5$, $d_1=2$, and $G$ has a vertex of degree~$2$, 
   then by Lemma~\ref{lem:all-d1}, $i(G)\le \frac{5}{9}\abs{V(G)}=(1-t)\abs{V(G)}$.
   Hence, we may assume that $\Delta\neq 5$, or $d_1\neq 2$, or $G$ has no vertex of degree~$2$.

    Let $v$ be a vertex of degree at least~$2$.    
    If $d_1\notin \{\lceil \Delta/2\rceil,\lfloor\Delta/2\rfloor\}$
    or $d_G(v)<\Delta$,
    then 
    by Lemma~\ref{lem:select2}, 
    $\frac{n_0(G-N_G[v])+1}{d_G(v)}\le \frac{\lfloor \Delta^2/4\rfloor}{\Delta}$, which implies that $i(G)\le (1-t)\abs{V(G)}$ as in the previous computation.
    Therefore, we may assume that $d_1\in\{\lceil\Delta/2\rceil,\lfloor\Delta/2\rfloor\}$
    and every vertex of degree at least~$2$ has degree exactly~$\Delta$.

    Then $G-X$ is $(\Delta-d_1)$-regular.
    Since $G$ is not special, $G-X$ is not a complete graph.

    If $G-X$ is not an odd cycle, then by Lemma~\ref{lem:all-d1}, \[ i(G)\le \left(1-\frac{\Delta-1}{(d_1+1)(\Delta-d_1)}\right)\abs{V(G)}.\] 
    If $\Delta\ge 6$, then 
    \begin{align*} 
        \frac{1}{t}-\frac{(d_1+1)(\Delta-d_1)}{\Delta-1}
        &\ge  \frac{(\Delta^2-1)/4+\Delta}{\Delta}-
        \frac{(\Delta+1)^2/4}{\Delta-1}
        \\ 
        &=\frac{
        \left( 
            (\Delta^3+3\Delta^2-5\Delta+1)-
            (\Delta^3+2\Delta^2+\Delta) 
        \right)}
        {4\Delta(\Delta-1)}\\
        &=\frac{ (\Delta^2-6\Delta+1)}{4\Delta(\Delta-1)} >0 
    \end{align*}
    and therefore 
    \[ i(G)\le \left(1-\frac{\Delta-1}{(d_1+1)(\Delta-d_1)}\right)\abs{V(G)}\le (1-t)\abs{V(G)}.\]
    If $\Delta=5$, then $d_1\in\{2,3\}$ and so 
    $\frac{\Delta-1}{(d_1+1)(\Delta-d_1)}\ge \frac{4}{9}=t$.
    If $\Delta=4$, then $d_1=2$ and 
    $\frac{\Delta-1}{(d_1+1)(\Delta-d_1)}=\frac{1}{2}=t$.

    Now we may assume that $G-X$ is an odd cycle. 
    Since $G$ is not special, 
    $\Delta=5$.
    Since $G-X$ is $(\Delta-d_1)$-regular, we have 
    $d_1=3$.
    As $G-X$ is an odd cycle, $G-X$ has an independent set $S$ of size $(\abs{V(G-X)}-1)/2 = \frac{\abs{V(G)}}{2(d_1+1)}-\frac12$.
    As in the proof of Lemma~\ref{lem:all-d1}, 
    $S\cup N_G(V(G)-X-S)$ is an independent dominating set
    and so $i(G)\le 
    \abs{S}+d_1\abs{V(G)-X-S}=\frac{\abs{V(G)}}{2(d_1+1)}-\frac12+d_1 (\frac{\abs{V(G)}}{2(d_1+1)}+\frac12)
    = \abs{V(G)}/2 + \frac{d_1-1}{2}=\abs{V(G)}/2+1$.
    Since $G$ is not special, the length of $G-X$ is at least $5$ and $\abs{V(G)}\ge 5(d_1+1)=20$, and therefore 
    $\abs{V(G)}/2+1\le \frac{5}{9}\abs{V(G)}$.
\end{proof}

\section*{Acknowledgement}

We thank Ilkyoo Choi for helpful discussions.

\bibliographystyle{amsplain}
\bibliography{ref} % see ref.bib for bibliography management

\end{document}